\documentclass[12pt,a4paper]{amsart}
\usepackage[bottom]{footmisc}
\usepackage{mathrsfs}
\usepackage{amsmath}
\usepackage{amsthm}
\usepackage{amsfonts}
\usepackage{latexsym}
\usepackage{graphicx}
\usepackage{hyperref}
\usepackage{amssymb}
\usepackage{epsf}
\usepackage{float}
\usepackage{fancyhdr}
\allowdisplaybreaks \makeatletter
\def\rightharpoonfill@{\arrowfill@\relbar\relbar\rightharpoonup}
\DeclareRobustCommand{\overrightharpoon}{\mathpalette{\underarrow@\rightharpoonfill@}}
\makeatother

\allowdisplaybreaks

\begin{document}
\newcommand{\beq}{\begin{equation}}
\newcommand{\eb}{\begin{equation}}
\newcommand{\eneq}{\end{equation}}
\newcommand{\ee}{\end{equation}}
\newtheorem{thm}{Theorem}[section]
\newtheorem{coro}[thm]{Corollary}
\newtheorem{lem}[thm]{Lemma}
\newtheorem{prop}[thm]{Proposition}
\newtheorem{defi}[thm]{Definition}
\newtheorem{rem}[thm]{Remark}
\newtheorem{cl}[thm]{Claim}
\title{Analytic solutions for the approximated Kantorovich mass
transfer problems by $p$-Laplacian approach}
\author{Yanhua Wu$^{1}$\ \ \ \ Xiaojun Lu$^2$}
\thanks{Corresponding author: Xiaojun Lu, Department of Mathematics \& Jiangsu Key Laboratory of
Engineering Mechanics, Southeast University, 210096, Nanjing, China}
\date{}
\maketitle
\pagestyle{fancy}                   % ÉèÖÃҳü
\lhead{Y. Wu and X. Lu} \rhead{Kantorovich mass transfer problem} %\rhead{\small\leftmark}
\begin{center}
1. Department of Sociology, School of Public Administration, Hohai
University, 211189, Nanjing, China\\
2. Department of Mathematics \& Jiangsu Key Laboratory of
Engineering Mechanics, Southeast University, 210096, Nanjing,
China\\
%2. Jiangsu Testing Center for Quality of Construction Engineering
%Co., Ltd, 210028, Nanjing, China
%2. Faculty of Science and Technology, Federation University
%Australia, Ballarat, VIC 3350, Australia
\end{center}
\vspace{.5cm}

\begin{abstract}
This manuscript discusses the approximation of a global maximizer of
the Kantorovich mass transfer problem through the approach of
$p$-Laplacian equation. Using an approximation mechanism, the primal
maximization problem can be transformed into a sequence of
minimization problems. By applying the canonical duality theory, one
is able to derive a sequence of analytic solutions for the
minimization problems. In the final analysis, the convergence of the
sequence to a global maximizer of the primal Kantorovich problem
will be demonstrated.
\end{abstract}
\section{Introduction}
Complementary variational method has been applied in the study of
finite deformation by Hellinger since the beginning of the 20th
century. During the last few years, considerable effort has been
taken to find minimizers for non-convex strain energy functionals
with a double-well potential. In this respect, Ericksen bar is a
typical model for the research of elastic phase transitions. In
\cite{G6}, R. W. Ogden et al. treated the Ericksen bar as a 1-D
smooth compact manifold and discussed two classical loading devices,
namely, hard device and soft device, by introducing a distributed
axial body force. By applying the canonical duality method, the
authors characterized the local energy extrema and the global energy
minimizer for both hard device and soft device. This method proved
to be very efficient in solving lots of open problems in the
mechanical fields such as non-convex optimal design and control,
nonlinear stability analysis of finite deformation, nonlinear
elastic theory with residual strain, existence results for Nash
equilibrium points of non-cooperative
games etc. Interested readers can refer to \cite{G1,G6,LU} for more details.\\

This paper mainly addresses the Kantorovich problem in higher
dimensions. Let $\Omega=\mathbb{B}(O_1,R_1)$ and
$\Omega^*=\mathbb{B}(O_2,R_2)$ denote two open balls with centers
$O_1$ and $O_2$, radii $R_1$ and $R_2$ in the Euclidean space
$\mathbb{R}^n$, respectively, and we denote $U:=\Omega\cup\Omega^*.$
Here we focus on the following two representative cases:
\begin{itemize}\item $\Omega=\mathbb{B}(O_1,R_1)$,
$\Omega^*=\mathbb{B}(O_2,R_2)$, $\Omega\cap\Omega^*=\emptyset$;
\item $\Omega=\mathbb{B}(O_1,R_1)$,
$\Omega^*=\mathbb{B}(O_1,R_2)$, $R_1\neq R_2$.
\end{itemize}
Let $f^+$ and $f^-$ be two nonnegative density functions in $\Omega$
and $\Omega^*$, respectively, and satisfy the normalized balance
condition
\[
\int_\Omega f^+dx=\int_{\Omega^*}f^-dx=1.
\]
For convenience's sake, let $f:=f^+-f^-$. First, let the admissible
set $\mathscr{A}$ be defined as
\[
\mathscr{A}:=\Big\{\phi\in W_0^{1,\infty}(U)\cap
C(\overline{U})\Big|\ \|\nabla\phi\|_{L^\infty}\leq 1, \phi\
\text{radially symmetric},\ \phi=0\ \text{on}\
\overline{\Omega\cap\Omega^*}\Big\},
\]
where $W_0^{1,\infty}(U)$ is a Sobolev spaces. The aim is to find an
analytic global maximizer (so-called {\it Kantorovich potential})
$u\in\mathscr{A}$ for the Kantorovich problem in the following form,
\beq(\mathcal{P}):
K[u]=\displaystyle\max_{w\in\mathscr{A}}\Big\{K[w]:=\int_{U}wfdx\Big\}.\eneq

In this paper, we consider the Kantorovich problem through a
$p$-Laplacian approach by introducing an approximation of the primal
($\mathcal{P}$) \cite{Evans2},
\begin{equation}(\mathcal{P}^{(p)}):
\displaystyle\min_{w_p\in\mathscr{A}}\Big\{I^{(p)}[w_p]:=\int_U
\Big(H^{(p)}(\nabla w_p)-w_pf\Big)dx\Big\},
\end{equation}
where $p>2$ and $H^{(p)}:\mathbb{R}^n\to\mathbb{R}^+$ is defined as
$$H^{(p)}(\gamma):=|\gamma|^p/p,$$
and $I^{(p)}$ is called the {\it potential energy functional}. It's
evident that $$
-\displaystyle\lim_{p\to+\infty}\displaystyle\min_{w_p\in\mathscr{A}}\{I^{(p)}[w_p]\}=\displaystyle\max_{w\in\mathscr{A}}\{K[w]\}.
$$ Consequently, once a function $\bar{u}_p$ satisfying
$I^{(p)}[\bar{u}_p]=\displaystyle\min_{w_p\in\mathscr{A}}\{I^{(p)}[w_p]\}$
is obtained, then it will help find out an analytic Kantorovich
potential
$u=\displaystyle\lim_{p\to+\infty}\bar{u}_p$ in the $L^\infty$ sense, which maximizes the primal problem ($\mathcal{P}$).\\

By variational calculus, one derives a corresponding Euler-Lagrange
equation for $(\mathcal{P}^{(p)})$, \beq
\begin{array}{ll}\displaystyle {\rm div}(|\nabla u_p|^{p-2}\nabla u_p)+f=0,& \ \text{\rm in}\
U\setminus\{\overline{\Omega\cap\Omega^*}\},
\end{array}\eneq
equipped with the Dirichlet boundary condition. For the integer
case, $p=1$, by variational calculus, one obtains the mean curvature
operator; $p=2$, one has the Laplace operator(see \cite{Evans4}).
For $p=n$, one derives the $n$-harmonic equation which is invariant
under M\"{o}bius transformation. While for the fractional case, such
as $p=3/2$, $p-$Laplacian describes the flow through porous media.
And glaciologist usually study the case $p\in(1,4/3]$. For more
background materials, please refer to \cite{LIONS}.

Clearly, (3) is a nonlinear $p$-Laplacian problem which is difficult
to solve by the direct approach \cite{JH,Evans4,LIONS}. However, by
the canonical duality theory, one is able to demonstrate the
existence and uniqueness of the solution for the nonlinear
differential equation, which establishes the equivalence between the
global minimizer of ($\mathcal{P}^{(p)}$)
and the solution of Euler-Lagrange equation (3).\\

At the moment, we would like to introduce the main theorems.
\begin{thm}
For any positive density functions $f^+\in C(\overline{\Omega})$ and
$f^-\in C(\overline{\Omega^*})$ satisfying the normalized balance
condition, there exists a unique solution $\bar{u}_p\in\mathscr{A}$
for the Euler-Lagrange equation (3), which is at the same time a
global minimizer for the approximation problem
($\mathcal{P}^{(p)}$). In particular, let
\[
E_p(x):=\displaystyle x^{(2p-2)/(p-2)}, x\in[0,1],
\]
and $E_p^{-1}$ stands for the inverse of $E_p$, then one has
\begin{itemize}
\item $\Omega=\mathbb{B}(O_1,R_1)$,
$\Omega^*=\mathbb{B}(O_2,R_2)$, $\Omega\cap\Omega^*=\emptyset$.
$\bar{u}_p$ can be represented explicitly as
\[
\bar{u}_p(r)= \left\{
\begin{array}{lll}
\displaystyle\int^{r}_{R_1}F(\rho)\rho/E_p^{-1}(F^2(\rho)\rho^2)d\rho,&
r\in[0,R_1],\\
\\
\displaystyle\int^{r}_{R_2}G(\rho)\rho/E_p^{-1}(G^2(\rho)\rho^2)d\rho,&
r\in[0,R_2],
\end{array}
\right.
\]
where $F$ and $G$ are defined as
\[
\left\{
\begin{array}{lll}
F(r):=-\displaystyle\Gamma(n/2)/(2\pi^{n/2}
r^n)+\int_r^{R_1}f^+(\rho)\rho^{n-1}/r^n d\rho,&
r\in[0,R_1],\\
\\
G(r):=\displaystyle\Gamma(n/2)/(2\pi^{n/2}
r^n)-\int_r^{R_2}f^-(\rho)\rho^{n-1}/r^n d\rho,&
r\in[0,R_2].\\
\end{array}
\right.
\]
\\
\item $\Omega=\mathbb{B}(O_1,R_1)$,
$\Omega^*=\mathbb{B}(O_1,R_2)$, $R_1>R_2>0$. $\bar{u}_p$ can be
represented explicitly as
\[
\bar{u}_p(r)=
\int^{r}_{R_2}F_p(\rho)\rho/E_p^{-1}(F_p^2(\rho)\rho^2)d\rho,\
r\in[R_2,R_1],
\]
where
\[
F_p(r):=C_pR_2^n/{r^n}-\int^r_{R_2}f^+(\rho)\rho^{n-1}/r^n d\rho,
\]
and $C_p\in(0,\Gamma(n/2)/(2\pi^{n/2}R_2^n))$.
\\
\item $\Omega=\mathbb{B}(O_1,R_1)$,
$\Omega^*=\mathbb{B}(O_1,R_2)$, $R_2>R_1>0$. $\bar{u}_p$ can be
represented explicitly as
\[
\bar{u}_p(r)=
\int^{r}_{R_1}G_p(\rho)\rho/E_p^{-1}(G_p^2(\rho)\rho^2)d\rho,\
r\in[R_1,R_2],
\]
where
\[
G_p(r):=-D_pR_1^n/{r^n}+\int^r_{R_1}f^-(\rho)\rho^{n-1}/r^n d\rho,
\]
and $D_p\in(0,\Gamma(n/2)/(2\pi^{n/2}R_1^n))$.
\end{itemize}
\end{thm}

\begin{thm}
For any positive density functions $f^+\in C(\overline{\Omega})$ and
$f^-\in C(\overline{\Omega^*})$ satisfying the normalized balance
condition, there exists a global maximizer for the Kantorovich
problem ($\mathcal{P}$).
\end{thm}

The rest of the paper is organized as follows. In Section 2, first
we introduce some useful notations which will simplify our proof
considerably. Then we apply the canonical dual transformation to
deduce a perfect dual problem ($\mathcal{P}^{(p)}_d$) corresponding
to $(\mathcal{P}^{(p)})$ and a pure complementary energy principle.
Next we apply the canonical duality theory to prove Theorem 1.1 and
Theorem 1.2.
\section{Proof of the main results}
\subsection{Some useful notations}

\begin{itemize}
\item $\overrightarrow{\theta_p}$ is given by
\[
\overrightarrow{\theta_p}(x)=(\theta_{p,1}(x),\cdots,\theta_{p,n}(x))=|\nabla
w_p|^{p-2}\nabla w_p.
\]
\item $\Phi^{(p)}:\mathscr{A}\to L^\infty(U)$ is a nonlinear
geometric mapping defined as
\[
\Phi^{(p)}(w_p):=|\nabla w_p|^2.
\]
For convenience's sake, denote $\xi_p:=\Phi^{(p)}(w_p).$ It is
evident that $\xi_p$ belongs to the function space $\mathscr{U}$
given by
\[
\mathscr{U}:= \Big\{\phi\in L^\infty(U)\Big| 0\leq\phi\leq1\Big\}.
\]
\item $\Psi^{(p)}:\mathscr{U}\to L^\infty(U)$ is a canonical energy
defined as
\[
\Psi^{(p)}(\xi_p):=\xi_p^{p/2}/p,
\]
which is a convex function with respect to $\xi_p$. For simplicity,
denote $\zeta_p:=\xi_p^{(p-2)/2}/2$, which is the G\^{a}teaux
derivative of $\Psi^{(p)}$ with respect to $\xi_p$. Moreover,
$\zeta_p$ is invertible with respect to $\xi_p$ and belongs to the
function space $\mathscr{W}$,
\[\mathscr{W}:=\Big\{\phi\in
L^\infty(U)\Big| 0\leq\phi\leq 1/2\Big\}.
\]
\item
$\Psi^{(p)}_\ast:\mathscr{W}\to L^\infty(U)$ is defined as
\[
\Psi^{(p)}_\ast(\zeta_p):=\xi_p\zeta_p-\Psi^{(p)}(\xi_p)=(1-2/p)2^{2/(p-2)}\zeta_p^{p/(p-2)}.
\]
\item $\lambda_p$ is defined as $\lambda_p:=2\zeta_p,$ and belongs
to the function space $\mathscr{V}$,
\[\mathscr{V}:=\Big\{\phi\in
L^\infty(U)\Big| 0\leq\phi\leq 1\Big\}.
\]
\end{itemize}
\subsection{Canonical duality techniques}
\begin{defi}
By Legendre transformation, one defines a Gao-Strang total
complementary energy functional $\Xi^{(p)}$,
\[
\Xi^{(p)}(u_p,\zeta_p):=\displaystyle\int_{U}\Big\{\Phi^{(p)}(u_p)\zeta_p-\Psi^{(p)}_\ast(\zeta_p)
-fu_p\Big\}dx.
\]
\end{defi}
Next we introduce an important {\it criticality criterium} for the
Gao-Strang total complementary energy functional.
\begin{defi}
$(\bar{u}_p, \bar{\zeta}_p)\in\mathscr{A}\times\mathscr{W}$ is
called a critical pair of $\Xi^{(p)}$ if and only if \beq
D_{u_p}\Xi^{(p)}(\bar{u}_p,\bar{\zeta}_p)=0, \eneq \beq
D_{\zeta_p}\Xi^{(p)}(\bar{u}_p,\bar{\zeta}_p)=0, \eneq where
$D_{u_p}, D_{\zeta_p}$ denote the partial G\^ateaux derivatives of
$\Xi^{(p)}$, respectively. \end{defi} Indeed, by variational
calculus, we have the following observation from (4) and (5).
\begin{lem}
On the one hand, for any fixed $\zeta_p\in\mathscr{W}$, $(3.4)$ is
equivalent to the equilibrium equation
\[
\begin{array}{ll}\displaystyle {\rm div}(\lambda_p \nabla\bar{u}_{p})+f=0,& \ \text{\rm in}\
U\setminus\{\overline{\Omega\cap\Omega^*}\}.\end{array}
\]
On the other hand, for any fixed $u_p\in\mathscr{A}$, (5) is
consistent with the constructive law
\[
\Phi^{(p)}(u_p)=D_{\zeta_p}\Psi^{(p)}_\ast(\bar{\zeta}_p).
\]
\end{lem}
Lemma 3.2.3 indicates that $\bar{u}_p$ from the critical pair
$(\bar{u}_p,\bar{\zeta}_p)$ solves the Euler-Lagrange equation (3).
\begin{defi}
From Definition 3.2.1, one defines the Gao-Strang pure complementary
energy $I^{(p)}_d$ in the form
\[
I^{(p)}_d[\zeta_p]:=\Xi^{(p)}(\bar{u}_p,\zeta_p),
\]
where $\bar{u}_p$ solves the Euler-Lagrange equation (3).
\end{defi}
To simplify the discussion, we use another representation of the
pure energy $I^{(p)}_d$ given by the following lemma.
\begin{lem} The
pure complementary energy functional $I^{(p)}_d$ can be rewritten as
\[
I^{(p)}_d[\zeta_p]=\displaystyle-\int_{U}\Big\{|\overrightarrow{\theta_p}|^2/(4\zeta_p)+(1-2/p)2^{2/(p-2)}\zeta_p^{p/(p-2)}\Big\},
\]
where $\overrightarrow{\theta_p}$ satisfies \beq {\rm
div}\overrightarrow{\theta_{p}}+f=0\ \text{in}\ U, \eneq equipped
with a hidden boundary condition.
\end{lem}
\begin{proof}
Through integrating by parts, one has
\[
\begin{array}{lll}
I^{(p)}_d[\zeta_p]&=&\displaystyle-\underbrace{\int_U\Big\{{\rm
div}(2\zeta_p \nabla\bar{u}_{p})+f\Big\}\bar{u}_pdx}_{(I)}\\
\\
&&-\underbrace{\int_U\Big\{\zeta_p|\nabla\bar{u}_{p}|^2+(1-2/p)2^{2/(p-2)}\zeta_p^{p/(p-2)}\Big\}dx.}_{(II)}\\
\\
\end{array}
\]
Since $\bar{u}_p$ solves the Euler-Lagrange equation (3), then the
first part $(I)$ disappears. Keeping in mind the definition of
$\overrightarrow{\theta_p}$ and $\zeta_p$, one reaches the
conclusion.
\end{proof}

With the above discussion, next we establish a variational problem
to the approximation problem ($\mathcal{P}^{(p)}$).
\begin{equation}
(\mathcal{P}_d^{(p)}):\displaystyle\max_{\zeta_p\in\mathscr{W}}\Big\{I^{(p)}_d[\zeta_p]=\displaystyle-\int_{U}\Big\{|\overrightarrow{\theta_p}|^2/(4\zeta_p)+(1-2/p)2^{2/(p-2)}\zeta_p^{p/(p-2)}\Big\}.
\end{equation}
Indeed, by calculating the G\^{a}teaux derivative of $I_d^{(p)}$
with respect to $\zeta_p$, one has \begin{lem} The variation of
$I_d^{(p)}$ with respect to $\zeta_p$ leads to the dual algebraic
equation (DAE), namely, \beq
|\overrightarrow{\theta_p}|^2=(2\bar{\zeta}_p)^{(2p-2)/(p-2)}, \eneq
where $\bar{\zeta}_p$ is from the critical pair
$(\bar{u}_p,\bar{\zeta}_p)$.
\end{lem}
Taking into account the notation of $\lambda_p$,  the identity (8)
can be rewritten as \beq
|\overrightarrow{\theta_p}|^2=E_p(\lambda_p)={\lambda}_p^{(2p-2)/(p-2)}.
\eneq It is evident $E_p$ is monotonously increasing with respect to
$\lambda\in[0,1]$.
\subsection{Proof of Theorem 1.1}
From the above discussion, one deduces that, once $\theta_p$ is
given, then the analytic solution of the Euler-Lagrange equation (3)
can be represented as \beq
\bar{u}_p(x)=\displaystyle\int^{x}_{x_0}\eta_p(t)dt, \eneq where
$x\in \overline{U}, x_0\in\partial U$, $\eta_p:=\theta_p/\lambda_p$.
Together with (9), one sees that $
\displaystyle\lim_{p\to+\infty}|\nabla\bar{u}_{p}|=1, $ which is
consistent with the a-priori estimate in \cite{Evans1}. Next we
verify that $\bar{u}_p$ is exactly a global minimizer for
($\mathcal{P}^{(p)}$) and $\bar{\zeta}_p$ is a global maximizer for
($\mathcal{P}^{(p)}_d$).
\begin{lem}(Canonical duality theory)
For any positive density functions $f^+\in C(\overline{\Omega})$ and
$f^-\in C(\overline{\Omega^*})$ satisfying the normalized balance
condition, there exists a unique radially symmetric solution
$\bar{u}_p\in\mathscr{A}$ for the Euler-Lagrange equations (3) with
Dirichlet boundary in the form of (10), which is a unique global
minimizer over $\mathscr{A}$ for the approximation problem
($\mathcal{P}^{(p)}$). And the corresponding $\bar{\zeta}_p$ is a
unique global maximizer over $\mathscr{W}$ for the dual problem
($\mathcal{P}_d^{(p)}$). Moreover, the following duality identity
holds, \beq
I^{(p)}(\bar{u}_p)=\displaystyle\min_{u_p\in\mathscr{A}}I^{(p)}(u_p)=\Xi^{(p)}(\bar{u}_p,\bar{\zeta}_p)=\displaystyle\max_{\zeta_p\in\mathscr{W}}I_d^{(p)}(\zeta_p)=I_d^{(p)}(\bar{\zeta}_p).
\eneq
\end{lem}
Lemma 3.2.7 shows that the maximization of the pure complementary
energy functional $I_d^{(p)}$ is perfectly dual to the minimization
of the potential energy functional $I^{(p)}$. Indeed, identity (11)
indicates there is no duality gap between them.
\begin{proof}
We divide our proof into three parts. In the first and second parts,
we discuss the uniqueness of $\theta_p$ for both cases. Global
extremum will be
studied in the third part. It is worth noticing that the first and second parts are similar to the proof of Theorem 1.2.\\
\\
{\it First Part: $\Omega=\mathbb{B}(O_1,R_1)$, $\Omega^*=\mathbb{B}(O_2,R_2)$, $\Omega\cap\Omega^*=\emptyset$}\\

{\bf (1) Discussion in $\Omega$}\\

Let $O_1=(a_1,a_2,\cdots,a_n)$. Actually, a radially symmetric
solution for the Euler-Lagrange equation (3) is of the form
\[
\overrightarrow{\theta_p}=\overline{F}_p(r)((x_1-a_1,\cdots,x_n-a_n))=\overline{F}_p\Big(\sqrt{\sum_{i=1}^n(x_i-a_i)^2}\Big)((x_1-a_1,\cdots,x_n-a_n)),
\]
where
\[
\overline{F}_p(r)=C_pR_1^n/r^n+\int_r^{R_1}f^+(\rho)\rho^{n-1}/r^n
d\rho
\]
is the unique solution of the differential equation
\[
\overline{F}_p'(r)+n\overline{F}_p(r)/r=-f^+(r)/r,\ \ \ r\in(0,R_1].
\]
Recall that $\bar{u}_p(R_1)=0$, as a result,
\[
\bar{u}_p(r)=\int_{R_1}^r\Big(R_1^nC_p+\int_{\rho}^{R_1}f^+(r)r^{n-1}
dr\Big)/\Big(\rho^{n-1}\lambda_p(\rho)\Big)d\rho, \ \ \ \
r\in(0,R_1].
\]
As a matter of fact, if $\bar{u}_p\in C[0,R_1]$, we have
\[
\displaystyle\lim_{\rho\to
0^+}\Big\{R_1^nC_p+\int_{\rho}^{R_1}f^+(r)r^{n-1} dr\Big\}=0,
\]
which indicates
\[
C_p=-\displaystyle\Gamma(n/2)/(2\pi^{n/2} R_1^n),
\]
from the normalized balance condition
$$\int_{\Omega}f^+(x)dx=2\pi^{n/2}/\Gamma(n/2)\int_0^{R_1}f^+(r)r^{n-1}dr=1.$$

{\bf (2) Discussion in $\Omega^*$}\\

Let $O_2=(b_1,b_2,\cdots,b_n)$. In fact, a radially symmetric
solution for the Euler-Lagrange equation (3) is of the form
\[
\overrightarrow{\theta_p}=\overline{G}_p(r)((x_1-b_1,\cdots,x_n-b_n))=\overline{G}_p\Big(\sqrt{\sum_{i=1}^n(x_i-b_i)^2}\Big)((x_1-b_1,\cdots,x_n-b_n)),
\]
where
\[
\overline{G}_p(r)=D_pR_2^n/r^n-\int_r^{R_2}f^-(\rho)\rho^{n-1}/r^n
d\rho
\]
is the unique solution of the differential equation
\[
\overline{G}_p'(r)+n\overline{G}_p(r)/r=f^-(r)/r,\ \ \ r\in(0,R_2].
\]
Recall that $\bar{u}_p(R_2)=0$, as a result,
\[
\bar{u}_p(r)=\int_{R_2}^r\Big(R_2^nD_p-\int_{\rho}^{R_2}f^-(r)r^{n-1}
dr\Big)/\Big(\rho^{n-1}\lambda_p(\rho)\Big)d\rho, \ \ \ \
r\in(0,R_2].
\]
Indeed, if $\bar{u}_p\in C[0,R_2]$, then by applying the similar
contradiction method as above, one has
\[
\displaystyle\lim_{\rho\to 0^+}\Big\{R_2^n
D_p-\int_{\rho}^{R_2}f^-(r)r^{n-1} dr\Big\}=0,
\]
which indicates
\[
D_p=\Gamma(n/2)/(2\pi^{n/2} R_2^n)
\]
from the normalized balance condition
\[
\int_{\Omega^*}f^-(x)dx=2\pi^{n/2}/\Gamma(n/2)\int_0^{R_2}f^-(r)r^{n-1}dr=1.
\]
\\
{\it Second Part: $\Omega=\mathbb{B}(O_1,R_1)$, $\Omega^*=\mathbb{B}(O_1,R_2)$, $R_1\neq R_2$}\\

{\bf (1) $R_1>R_2>0$}\\

Let $O_1=(a_1,a_2,\cdots,a_n)$. Actually, a radially symmetric
solution for the Euler-Lagrange equation (3) is of the form
\[
\overrightarrow{\theta_p}=F_p(r)((x_1-a_1,\cdots,x_n-a_n))=F_p\Big(\sqrt{\sum_{i=1}^n(x_i-a_i)^2}\Big)((x_1-a_1,\cdots,x_n-a_n)),
\]
where
\[
F_p(r)=C_pR_2^n/r^n-\int^r_{R_2}f^+(\rho)\rho^{n-1}/r^n d\rho
\]
is the unique solution of the differential equation
\[
F_p'(r)+nF_p(r)/r=-f^+(r)/r,\ \ \ r\in[R_2,R_1].
\]
Recall that $\bar{u}_p(R_2)=0$, consequently,
\[
\bar{u}_p(r)=\int_{R_2}^r\Big(R_2^nC_p-\int^{\rho}_{R_2}f^+(r)r^{n-1}
dr\Big)/\Big(\rho^{n-1}\lambda_p(\rho)\Big)d\rho, \ \ \ \
r\in[R_2,R_1].
\]
Let
\[
\tilde{F}(r):=1/R_2^n\int_{R_2}^rf^+(\rho)\rho^{n-1}d\rho, \ \ \ \
r\in[R_2,R_1].
\]
Since $f^+>0$, then $\tilde{F}\in C[R_2,R_1]$ is a strictly
increasing function with respect to $r\in[R_2,R_1]$ and consequently
is invertible. Let $\tilde{F}^{-1}$ be its inverse function, which
is also a strictly increasing function. From (9), we see that there
exists a unique piecewise continuous function $\lambda_p(x)\geq0$.
Since
\[
\displaystyle\lim_{r\to
\tilde{F}^{-1}(C_p)}(-\tilde{F}(r)+C_p)R_2^n/(r^{n-1}\lambda_p(r))=0,
\]
thus $\bar{u}_p$ is continuous at the point $r=\tilde{F}^{-1}(C_p)$.
As a result, $\bar{u}_p\in C[R_2,R_1]$. Notice that
$\bar{u}_p(R_1)=0$ and we can determine the constant $C_p$ uniquely.
Indeed, let
\[
\mu_p(\rho,t):=\Big(R_2^nt-\int^{\rho}_{R_2}f^+(r)r^{n-1}
dr\Big)/\Big(\rho^{n-1}\lambda_p(\rho,t)\Big)
\]
and
\[
M_{k}(t):=\int^{\tilde{F}^{-1}(t)}_{R_2}\mu_p(\rho,t)d\rho+\int^{R_1}_{\tilde{F}^{-1}(t)}\mu_p(\rho,t)d\rho,
\]
where $\lambda_p(\rho,t)$ is from (9). It is evident that
$\lambda_p$ depends on $C_p$. As a matter of fact, it is easy to
check $M$ is strictly increasing with respect to $t$, which leads to
$$C_p=M_p^{-1}(0).$$
Furthermore, by a similar discussion as in \cite{LU}, we have
$$\lim_{k\to\infty} C_p=\tilde{F}((R_1+R_2)/2).$$

{\bf (2) $0<R_1<R_2$}\\

In fact, a radially symmetric solution for the Euler-Lagrange
equation (3) is of the form
\[
\overrightarrow{\theta_p}=G_p(r)((x_1-a_1,\cdots,x_n-a_n))=G_p\Big(\sqrt{\sum_{i=1}^n(x_i-a_i)^2}\Big)((x_1-a_1,\cdots,x_n-a_n)),
\]
where
\[
G_p(r)=-D_pR_1^n/r^n+\int^r_{R_1}f^-(\rho)\rho^{n-1}/r^n d\rho
\]
is the unique solution of the differential equation
\[
G_p'(r)+nG_p(r)/r=f^-(r)/r,\ \ \ r\in[R_1,R_2].
\]
Recall that $\bar{u}_p(R_1)=0$, as a result,
\[
\bar{u}_p(r)=\int_{R_1}^r\Big(-R_1^nD_p+\int^{\rho}_{R_1}f^-(r)r^{n-1}
dr\Big)/\Big(\rho^{n-1}\lambda_p(\rho)\Big)d\rho, \ \ \ \
r\in[R_1,R_2].
\]
Let
\[
\tilde{G}(r):=1/R_1^n\int_{R_1}^rf^-(\rho)\rho^{n-1}d\rho, \ \ \ \
r\in[R_1,R_2].
\]
Since $f^->0$, then $\tilde{G}\in C[R_1,R_2]$ is a strictly
increasing function with respect to $r\in[R_1,R_2]$ and consequently
is invertible. Let $\tilde{G}^{-1}$ be its inverse function, which
is also a strictly increasing function. From (9), we see that there
exists a unique piecewise continuous function $\lambda_p(x)\geq0$.
Since
\[
\displaystyle\lim_{r\to
G^{-1}(D_p)}(G(r)-D_p)R_1^n/(r^{n-1}\lambda_p(r))=0,
\]
thus $\bar{u}_p$ is continuous at the point $r=G^{-1}(D_p)$. As a
result, $\bar{u}_p\in C[R_1,R_2]$. Notice that $\bar{u}_p(R_2)=0$
and we can determine the constant $D_p$ uniquely. Indeed, let
\[
\eta_p(\rho,t):=\Big(-R_1^nt+\int^{\rho}_{R_1}f^-(r)r^{n-1}
dr\Big)/\Big(\rho^{n-1}\lambda_p(\rho,t)\Big)
\]
and
\[
N_{p}(t):=\int^{\tilde{G}^{-1}(t)}_{R_1}\eta_p(\rho,t)d\rho+\int^{R_2}_{\tilde{G}^{-1}(t)}\eta_p(\rho,t)d\rho,
\]
where $\lambda_p(\rho,t)$ is from (9). It is evident that
$\lambda_p$ depends on $C_p$. As a matter of fact, it is easy to
check $N_p$ is strictly increasing with respect to $t$, which leads
to
\[
D_p=N_p^{-1}(0).
\]
Furthermore, by a similar discussion as in \cite{LU}, we have
\[
\lim_{k\to\infty} D_p=\tilde{G}((R_1+R_2)/2).
\]
\\
{\it Third Part:}\\

On the one hand, for any test function $\phi\in\mathscr{A}$
satisfying $\nabla\phi\neq0$ a.e. in $U$, the second variational
form $\delta_\phi^2I^{(p)}$ with respect to $\phi$ is equal to\beq
\int_U\Big\{|\nabla \bar{u}_p|^{p-2}|\nabla\phi|^2+(p-2)|\nabla
\bar{u}_p|^{p-4}(\nabla \bar{u}_p\cdot\nabla\phi)^2\Big\}dx.\eneq On
the other hand, for any test function $\psi\in\mathscr{W}$
satisfying $\psi\neq0$ a.e. in $U$, the second variational form
$\delta_\psi^2I_d^{(p)}$ with respect to $\psi$ is equal to
\beq-\int_U\psi^2\Big\{|\overrightarrow{\theta_p}|^2/(2\bar{\zeta}_p^3)+1/(p-2)2^{p/(p-2)}\zeta_p^{(4-p)/(p-2)}\Big\}dx.
\eneq From (12) and (13), one deduces immediately that
\[
\delta^2_\phi I^{(p)}(\bar{u}_p)>0,\ \
\delta_\psi^2J_d^{(p)}(\bar{\zeta}_p)<0.
\]
Together with the uniqueness of $\overrightarrow{\theta_p}$
discussed in the first and second parts, the proof is concluded.
\end{proof}
Consequently, we reach the conclusion of Theorem 1.1 by summarizing
the above discussion.
\subsection{Proof of Theorem 1.2:}
According to Rellich-Kondrachov Compactness Theorem, since
$$\displaystyle\sup_{k}|\bar{u}_k|\leq C(R_1,R_2)$$ and
$$\displaystyle\sup_{k}|\nabla\bar{u}_{k}|\leq 1,$$ then, there exists a
subsequence(without any confusion, we still denote as)
$\{\bar{u}_{k}\}_{k}$ and $u\in W_0^{1,\infty}(U)\cap
C(\overline{U})$ such that \beq\bar{u}_{k}\rightarrow u\
(k\to\infty)\ \text{in}\ L^\infty(U),\eneq \beq\nabla\bar{u}_{k}\
\overrightharpoon{*}\ \nabla u\ (k\to\infty)\ \text{weakly\ $\ast$\
in}\ L^\infty(U).\eneq It remains to check that $u$ satisfies (5).
%From (20), one knows \beq\bar{u}_{\varepsilon_k}\rightarrow
%u\ (k\to\infty)\ \text{a.e.\ in}\ \Omega.\eneq According to
%Lebesgue's dominated convergence theorem,
%\[\displaystyle\int_{\Omega}f(y)dy=\lim_{k\to\infty}\int_{\Omega}\bar{u}_{\varepsilon_k}(y)dy=1.\]
From (19), one has
\[\|\nabla u\|_{L^\infty(U)}\leq\displaystyle\liminf_{k\to\infty}\|\nabla\bar{u}_{k}\|_{L^\infty(U)}\leq\sup_{k\to\infty}\|\nabla\bar{u}_{k}\|_{L^\infty(U)}\leq 1.\]
Consequently, one reaches the conclusion of Theorem 1.2 by
summarizing the above discussion.
\begin{rem}
Frankly speaking, the uniqueness of the global minimizer for the
primal problem does not hold when $U$ is a general Lipschitz domain.
As $p\to\infty$, we have the infinity harmonic equation
$$\displaystyle\sum_{j,k=1}^{n}\frac{\partial u}{\partial x_j}\frac{\partial u}{\partial x_k}\frac{\partial^2 u}{\partial x_j\partial
x_k}=f.$$ This equation has often been used in image processing and
optimal Lipschitz extensions.
\end{rem}
{\bf Acknowledgment}: This project is partially supported by US Air
Force Office of Scientific Research (AFOSR FA9550-10-1-0487),
Natural Science Foundation of Jiangsu Province (BK 20130598),
National Natural Science Foundation of China (NSFC 71673043,
71273048, 71473036, 11471072), the Scientific Research Foundation
for the Returned Overseas Chinese Scholars, Fundamental Research
Funds for the Central Universities on the Field Research of
Commercialization of Marriage between China and Vietnam (No.
2014B15214). This work is also supported by Open Research Fund
Program of Jiangsu Key Laboratory of Engineering Mechanics,
Southeast University (LEM16B06).

\end{document}